\documentclass{amsart}

\usepackage{amssymb, amscd, amsmath}
\usepackage[all]{xy}
\usepackage[obeyspaces]{url}

\newtheorem{theorem}{Theorem}[section]
\newtheorem{proposition}[theorem]{Proposition}
\newtheorem{corollary}[theorem]{Corollary}
\newtheorem{lemma}[theorem]{Lemma}

\theoremstyle{remark}
\newtheorem{remark}[theorem]{\bf Remark}
\newtheorem{algorithm}[theorem]{\bf Algorithm}

\newcommand{\cO}{\mathcal{O}}

\newcommand{\fD}{\mathfrak{D}}

\newcommand{\ii}{\textbf{i}}

\newcommand{\ff}{\widetilde{\varphi}}
\newcommand{\fp}{\widetilde{\psi}}

\newcommand{\mQ}{\mathbb{Q}_p}

\newcommand{\typezero}{\langle 0 \rangle}
\newcommand{\typelambda}{\langle \lambda \rangle}

\makeatletter
    
   \@addtoreset{table}{section}
  \makeatother
\title[\tiny An identification for Eisenstein polynomials over a $p$-adic field]
{An identification for Eisenstein polynomials over a $p$-adic field}

\author{Shun'ichi Yokoyama and Manabu Yoshida}
\address{
	Graduate School of Mathematics, Kyushu University,
	Fukuoka 819-0395, Japan
}
\email{s-yokoyama@math.kyushu-u.ac.jp\\m-yoshida@math.kyushu-u.ac.jp}
\thanks{The authors are supported by the Japan Society for the Promotion 
of Scientist Fellowships for Young Scientists.}
\date{}

\begin{document}

\maketitle

%%%%%%%%%%%%%%%%% アブストラクト %%%%%%%%%%%%%%%%%%%%%
%\begin{abstract}
%We give a criteria whether given two Eisenstein polynomials over a 
%$p$-adic field define the same extension. 
%In particular, we deal with extensions of degree $p$. 
%Our code of Magma \cite{Magma} is available at 
%\url{http://www2.math.kyushu-u.ac.jp/~m-yoshida/calculator.html}. 
%This is an English translation of part of \cite{Yokoyama12}. 
%\end{abstract}
%%%%%%%%%%%%%%%%%%%%%%%%%%%%%%%%%%%%%%%%%%%%%%%%%%%%%%

%%%%%%%%%%%%%%%%% イントロ %%%%%%%%%%%%%%%%%%%%%%%%%%%%%%%%%%%%%%%%%%%%%%%%%%%%%%%%%%%%%%%%%%%%
%\section{Introduction}

In this note, we give a criteria whether given two Eisenstein polynomials over a 
$p$-adic field define the same extension (Proposition \ref{Krasner}). 
In particular, we completely identify Eisenstein polynomials of degree $p$ (Theorem \ref{MainTheorem}). 
%Magma \cite{Magma} can also check the isomorphy of $p$-adic field extensions by 
%inner function \texttt{IsIsomorphic}. 
%However, if the degree of polynomials is huge, 
%then it does not work efficiently since the number of extensions are also huge. 
%Our theorem directly identify a given polynomial as the one in the list (Table \ref{Representative}) of a complete system of representatives of isomorphism classes of extensions 
%without Root Counting Algorithm. 
%Hence we can speed up to identify extensions. 
%Such speeding up is useful for calculating the Galois groups of local field extensions (Remark \ref{GaloisGroup}). 
This note is an English translation of a part of \cite{Yokoyama12}. 

%%%%%%%%%%%%%%%%%%%%%%%%%%%%%%%%%%%%%%%%%%%%%%%%%%%%%%%%%%%%%%%%%%%%%%%%%%%%%%%%%%%%%%%%%%%%%%%%%%%%%%%%%%%%
\section{Eisenstein polynomials and ramification theory}\label{Chap:Ramification}

In Section \ref{Chap:General}, we consider general Eisenstein polynomials. 
In Section \ref{Chap:DegreeP}, we precisely investigate Eisenstein polynomials of degree $p$ over $\mathbb{Q}_p$. 

%%%%%%%%%%%%%%%%%%%%%%%%%%%%%%%%%%%%%%%%%%%%%%%%%%%%%%%%%%%%%%%%%%%%%%%%%%%%%%%%%%%%%%%%%%%%%%%%%%%%%%%%%%%%%%%%%%%%%%%%%%%%%%
\subsection{General Eisenstein polynomials}\label{Chap:General}

In this subsection, we assume $K$ is a finite extension of $\mathbb{Q}_p$\footnote{
The results in Subsection \ref{Chap:General} hold for any complete discrete valuation field $K$ with perfect residue field of characteristic $p$ 
and any finite separable extension $L/K$. 
}. 
We fix an algebraic closure $\overline{K}$ of $K$ and 
we assume throughout that all algebraic extensions of $K$ under discussion 
are contained in $\overline{K}$. 
We denote by $\mathcal{O}_K$ the valuation ring of $K$ and by $v_K$ the valuation on $\overline{K}$ such that $v_K(K^{\times})=\mathbb{Z}$.  
Let $L$ be a finite separable extension of $K$.  
We denote by $\mathcal{O}_L$ the integral closure of $\mathcal{O}_K$ in $L$. 
There exists an element $\alpha \in \cO_L$ such that $\cO_L = \cO_K[\alpha]$ 
(the existence of such an element is proved in 
\cite{Serre79}, Chap.\ III, Sect.\ 6, Prop.\ 12). 
Put $H=\mathrm{Hom}_K(L,\overline{K})$.  
The order function $\ii_{L/K}$ is defined on $H$ by 
%%%%%%% i_L/K(σ)の定義 %%%%%%%%
\[\ii_{L/K}(\sigma) = v_K(\sigma(\alpha) - \alpha)\]
for any $\sigma \in H$. 
This function is independent of the choice of $\alpha$. 
The $i$th \emph{lower numbering ramification set $H_{(i)}$ of $H$} 
are defined for a real number $i \geq 0$ by 
%%%%%%% 下付き分岐の定義 %%%%%%%
\[H_{(i)}=\{ \sigma \in H\ |\ \ii_{L/K}(\sigma) \geq i \}.\]
The transition function $\ff_{L/K}:\mathbb{R}_{\geq 0} \to 
\mathbb{R}_{\geq 0}$ of $L/K$ is defined by 
%%%%%%% Herbrand 関数の定義 %%%%%
\[\ff_{L/K}(i)=\int_0^i \# H_{(t)} dt\]
for any real number $i \geq 0$, 
where $\# H_{(t)}$ is the cardinality of $H_{(t)}$. 
Its inverse function is denoted by $\fp_{L/K}$. 
Then the $u$th \emph{upper numbering ramification set $H^{(u)}$ of $H$} 
are defined for a real number $u \geq 0$ by 
%%%%%%% 上付き分岐の定義 %%%%%%
\[ H^{(u)}=H_{(i_0)},\quad i_0:=\widetilde{\psi}_{L/K}(u) \]
%%%%%%%%%%%%%%%%%%%%%%%%%%%%%
A \emph{ramification break} is a real number $i$ (resp.\ $u$) such that $H_{(i)} \not= H_{(i + \varepsilon)}$ (resp.\ $H^{(u)} \not= H^{(u + \varepsilon)}$) for any $\varepsilon>0$.  
%%%%%%%  i_L/K と u_L/K %%%%%%%%%%
Denote the largest lower (resp.\ upper) numbering ramification break by 
\[ i_{L/K}=\inf \{ i \in \mathbb{R}\ |\ H_{(i)}=1 \},\quad 
u_{L/K}=\inf \{ u \in \mathbb{R}\ |\ H^{(u)}=1 \}. \]
%%%%%%%%%%%%%%%%%%%%%%%%%%%%%%%%%%
\begin{remark}
If $L/K$ is a Galois extension, then our filtration $H^{(u)}$ coincides with 
the filtration shifted by one defined in \cite{Serre79}, Chapter IV. 
\end{remark}
%%%%%%%%%%%%%%%%%%%%%%%%%%%%%%%
\begin{proposition}[\cite{Deligne84}, Prop.\ A.6.1]\label{Hilbert}
Let $\fD_{L/K}$ be the different of $L/K$. 
Then we have 
\[ u_{L/K} = i_{L/K} + v_K(\fD_{L/K}). \]
In particular, if $L/K$ has only one ramification break, 
then the above shows 
\[ u_{L/K} = e \cdot i_{L/K} = e/(e-1) \cdot v_K(\fD_{L/K}), \]
where $e$ is the ramification index of $L/K$. 
\end{proposition}
\begin{remark}
The Galois case is proved by \cite{Fontaine85}, Proposition 1.3. 
\end{remark}
%%%%%%%%%%%%%%%%%%%%%%%%%%%%%%%
\begin{lemma}[\cite{Deligne84}, Prop.\ A.6.2]\label{Herbrand}
Let $L$ be a finite separable extension of $K$. 
Put $H=\mathrm{Hom}(L,\overline{K})$. 
Choose an element $\alpha \in \cO_L$ such that $\cO_L=\cO_K[\alpha]$. 
Let $f$ be the minimal polynomial of $\alpha$ over $K$ and $\beta$ an element of $\Omega$. 
Put $i = \sup_{\sigma \in H} v_K(\sigma (\alpha) - \beta)$ and 
$u= v_K(f(\beta))$. 
Then we have
\[ u = \ff_{L/K}(i),\quad \fp_{L/K}(u)=i. \]
\end{lemma}
\begin{remark}
The numbering of the ramification filtration in \cite{Deligne84} is different from ours. 
We adopt the numbering in \cite{Fontaine85} since it is suitable for Proposition \ref{Krasner}, which is repeatedly used in this paper. 
\end{remark}
%%%%%%%%%%%%%%%%%%%%%%%%%%%%%%
Let $E_K^e$ be the set of all Eisenstein polynomials of degree $e$ over $K$. 
For two polynomials $f$, $g \in E_K^e$, we put
\[ v_K(f,g)=\min_{0 \leq i \leq e-1} \{ v_K(a_i - b_i) + \frac{i}{e} \}. \]
Then we have $v_K(f,g)=v_K(f(\pi_g))$ for any root $\pi_g$ of $g$, and 
$v_K(\cdot,\cdot)$ defines an ultrametric on $E_K^e$ (\emph{cf}.\ \cite{Krasner66}, \cite{Pauli01}). 
For each $f \in E_K^e$, we put $L_f=K[X]/(f)$ and $u_f=u_{L_f/K}$. 
For any $f,g \in E_K^e$, we define an equivalence $f \sim g$ on $E_K^e$ 
by the existence of a $K$-isomorphism $L_f \cong L_g$. 
%%%%%%%%%%%%%%%%%%%%%%%%%%%%%%%
\begin{proposition}\label{Krasner}
Let $f, g \in E_K^e$. If $v_K(f,g) > u_f$, then we have $f \sim g$. 
\end{proposition}
%%%%%%%%%%%%%%%%%%%%%%%%%%%%%%%
\begin{proof}
Take a root $\pi_g$ of $g$ and choose a root $\pi_f$ of $f$ 
such that $v_K(\pi_g - \pi_f)$ is the maximum. 
Put $u_f=u_{L_f/K}$, $i_f=i_{L_f/K}$, $\widetilde{\psi}_f=\widetilde{\psi}_{L_f/K}$. 
By assumption, we have $v_K(f,g)=v_K(f(\pi_g))  > u_f$. 
Note that $\mathcal{O}_{L_f}=\mathcal{O}_K[\pi_f]$. 
Mapping the equation by $\widetilde{\psi}_{f}$ gives an inequality
\[ v_K(\pi_g - \pi_f) = \widetilde{\psi}_{f}(v_K(f(\pi_g))) 
> \widetilde{\psi}_{f}(u_f) = i_f = \sup_{\sigma \in H, \sigma \not= 1} v_K(\sigma (\pi_f) -\pi_f)  \]
by Lemma \ref{Herbrand}. 
We have $L_f \cong K(\pi_f) \subset K(\pi_g) \cong L_g$ by Krasner's lemma, 
Since their degrees are the same, we obtain an isomorphism $L_f \cong L_g$.  
\end{proof}
%%%%%%%%%%%%%%%%%%%%%%%%%%%%%%%%%%%%%
\begin{remark}
The case where $L_f/K$ is a Galois extension is proved in \cite{Yoshida11}, Proposition 3.1. 
\end{remark}
%%%%%%%%%%%%%%%%%%%%%%%%%%%%%%%%%%%%%

%%%%%%%%%%%%%%%%%%%%%%%%%%%%%%%%%%%%%%%%%%%%%%%%%%%%%%%%%%%%%%%%%%%%%%%%%%%%%%%%%%%%%%%%%%%%%%%
\subsection{Degree $p$}\label{Chap:DegreeP}

In this subsection, we assume that $p$ is odd\footnote{
We can easily check the isomorphy of quadratic extensions, 
so that we consider only odd primes. 
}. 
We assume that the base field is $\mathbb{Q}_p$ and denote the $p$-adic valuation by $v_p$. 

%%%%%%%%%%%%%%%%%%%%%%%%%%%%%%%%%%%%%%%
\begin{proposition}[\cite{Amano71}, Thm.\ 6 and 7]\label{Representative}
Suppose that $p$ is odd. 
Table \ref{TableDegreeP} gives exactly one polynomial for each isomorphism class 
of totally ramified extension of $\mQ$ of degree $p$. 
In the table, we put $d_f=v_p(\fD_{L_f/\mQ})$. 
\begin{table}[htbp]
\newlength{\myheight}
\setlength{\myheight}{1.3cm}
\newlength{\myheighta}
\setlength{\myheighta}{1.6cm}
\begin{tabular}{|c|c|c|c|}
\hline
$\mathrm{Family}$ & $\mathrm{Parameter}$ & $d_f$ & $u_f$ \\
\hline
\parbox[c][\myheighta][c]{0cm}{} 
$X^p + a p X^{\lambda} + p$ & $\begin{matrix} 1 \leq a \leq p-1 \\
1 \leq \lambda \leq p-1 \\
(\lambda,a) \not= (p-1,p-1) \end{matrix}$ & $1 + \cfrac{\lambda-1}{p}$ 
& $1+ \cfrac{\lambda}{p-1}$ \\
\hline
\parbox[c][\myheight][c]{0cm}{} 
$X^p - p X^{p-1} + (1 + a p) p$  & $0 \leq a \leq p-1$ & $1 + \cfrac{p-2}{p}$ & $2$ \\
\hline
\parbox[c][\myheight][c]{0cm}{} 
$X^p + (1 + a p) p$ & $0 \leq a \leq p-1$ & $1 + \cfrac{p-1}{p}$ & $2 + \cfrac{1}{p-1}$ \\
\hline
\end{tabular}
\caption{A complete system of representatives of $E_{\mQ}^p/\sim$}\label{TableDegreeP}
\end{table}
\end{proposition}
%%%%%%%%%%%%%%%%%%%%%%%%%%%%%%%%%
\begin{remark}[\cite{Jones}, Prop.\ 2.3.1 for details]\label{GaloisGroup} 
(i) The Galois group of a polynomial $f$ of the first type in Table \ref{TableDegreeP} 
is a semi-direct product $C_p:C_{d_2}$, where $d_2 = (p-1)/\mathrm{gcd}((p-1)/m,g)$, $g=\mathrm{gcd}(p-1,\lambda)$ and 
$m$ is the order of $a \lambda$ in $\mathbb{F}_p^{\times}$. 
Moreover, its inertia group is $C_p:C_{d_1}$, 
where $d_1=(p-1)/g$. 
The second type in Table \ref{TableDegreeP} is the only case 
that $L_f/\mQ$ is a cyclic. 
The Galois group and its inertia subgroup of the third type in Table \ref{TableDegreeP} are $C_p:C_{p-1}$. 

\noindent
(ii) More precisely, in \cite{Amano71}, the explicit description of the Galois closure of $L_f/\mathbb{Q}_p$ 
as $\mathbb{Q}_p(\pi_f,\gamma)$ where $\gamma^{p-1} \in \mathbb{Q}_p$. 
An algorithm for computing the automorphism group of a finite extension $L/\mathbb{Q}_p$ 
has been implemented in Magma as the inner function \texttt{AutomorphismGroup}($L$,$\mathbb{Q}_p$), 
where the output is given as a subgroup of the symmetric group $S_p$. 
Hence we can explicitly calculate the Galois group of $f$.  
\end{remark}
%%%%%%%%%%%%%%%%%%%%%%%%%%%%%%%%%

Let $f = X^p + a_{p-1} X^{p-1} + \cdots + a_1 X + a_0 \in E_{\mQ}^p$. 
We say that $f$ is \emph{of type} $\typezero$ if 
$v_p(a_i) \geq 2$ for any $i$. 
If $v_p(a_i)=1$ for some $i$, then we put 
$\lambda:=\min \{1 \leq i \leq p-1 |\ v_p(a_i)=1 \}$. 
In this case, we say that $f$ is \emph{of type} $\typelambda$. 
Then we see
\[
d_f = 
\begin{cases} 1 + (\lambda - 1)/p & \text{$f$ $\mathrm{is\ of\ type}$ $\typelambda$}\\
1 + (p-1)/p & \text{$f$ $\mathrm{is\ of\ type}$ $\typezero$}.\\
\end{cases}
\]
The type of $f$ depends only on its equivalence class since 
$d_f$ does also. 
%%%%%%%%%%%%%%%%%%%%%%%%%%%%%%%%%%%
\begin{lemma}[\cite{Amano71}, Lemma 1]\label{OneBreak}
For any totally ramified extension $L/K$ of degree $p$ 
has only one ramification break. 
\end{lemma}
%%%%%%%%%%%%%%%%%%%%%%%%%%%%%%%%%%%
\noindent
By Proposition \ref{Hilbert} and Lemma \ref{OneBreak}, 
we have 
\[
u_f = 
\begin{cases}
1 + \lambda/(p-1) & \text{$f$ is of type $\typelambda$}\\
2 + 1/(p-1) & \text{$f$ is of type $\typezero$}.
\end{cases}
\]

%%%%%%%%%%%%%%%%%%%%%%%%%%%%%%%%%%%%%%%
\begin{proposition}\label{Prop:Type0TypeLambda}
For $f=\sum_i a_i x^i$, $g=\sum_i b_i x^i \in E_{\mathbb{Q}_p}^p$, 
if one of the following conditions is satisfied, then we have $f \sim g$. 

\noindent
$\mathrm{(i)}$ Both $f$ and $g$ are of type $\typelambda$, 
$\lambda < p-1$ and $v_p(a_i - b_i) \geq 2$ $(i=0,\lambda)$. 

\noindent
$\mathrm{(ii)}$ Both $f$ and $g$ are of type $\langle p-1 \rangle$ and 
$v_p(a_i - b_i) \geq 3$ $(i=0,p-1)$. 

\noindent
$\mathrm{(iii)}$ Both $f$ and $g$ are of type $\typezero$ and $v_p(a_i-b_i) \geq 3$ $(i=0,1)$.  
\end{proposition}
\begin{proof}
In each case, it is enough to show $v_p(f,g) > u_f$ by Proposition \ref{Krasner}. 
In case (i), by assumption, we have
\[ v_p(f,g) = \min_{1 \leq i \leq p-1} \left\{ v_p(a_i - b_i) + \frac{i}{p} \right\} 
\geq \min \left\{ 2, 1 + \frac{\lambda+1}{p} \right\} > 1 + \frac{\lambda}{p-1} = u_f. \]
Similarly, in case (ii), we have
\[ v_p(f,g) = \min_{1 \leq i \leq p-1} \left\{ v_p(a_i - b_i) + \frac{i}{p} \right\} \geq 2 + \frac{1}{p} > 1 + \frac{\lambda}{p-1} = u_f. \]
Finally, in case (iii), note that $v_p(a_i - b_i) \geq 2$ for any $i$, 
so that we have
\[ v_p(f,g) = \min_{1 \leq i \leq p-1}\left\{v_p(a_i - b_i) + \cfrac{i}{p} \right\} 
\geq 2 + \frac{2}{p} > 2 + \frac{1}{p-1} = u_f, \]
where the last inequality follows from the oddness of $p$. 
\end{proof}
%%%%%%%%%%%%%%%%%%%%%%%%%%%%%%%%%%%%%%%
\begin{corollary}\label{Cor:Type0TypeLambda}
Let $f=\sum_i a_i x^i \in E_{\mathbb{Q}_p}^p$. 

\noindent
$\mathrm{(i)}$ If $f$ is of type $\typelambda$, then $f \sim x^p + a_{\lambda} x^{\lambda} + a_0$. 

\noindent
$\mathrm{(ii)}$ If $f$ is of type $\typezero$ , then $f \sim x^p + a_1 x + a_0$. 
Furthermore, if $v_p(a_1) \not= 2$, then $f \sim x^p + a_0$. 
\end{corollary}
%%%%%%%%%%%%%%%%%%%%%%%%%%%%%%%%%%%%%%%
\noindent
To consider the case where $f$ is of type $\typezero$ and $v_p(a_1)=2$, 
we need some devise. 
%%%%%%%%%%%%%%%%%%%%%%%
\begin{lemma}\label{Lem:Type0}
Let $f \in E_{\mathbb{Q}_p}^p$ and $\pi$ be a root of $f$. 

\noindent
$\mathrm{(i)}$ For $u \in \mathbb{Z}_p^{\times}$, if we put $\pi'=\pi + u \pi^2$, 
then we have $\mathbb{Q}_p(\pi')=\mathbb{Q}_p(\pi)$. 

\noindent
$\mathrm{(ii)}$ Take the minimal polynomial $g \in E_{\mathbb{Q}_p}^p$ of $\pi'$.
Then we have
\[ g(ux^2 + x) = -u^p f(x) f(-x-u^{-1}). \]
\end{lemma}
\begin{proof}
(i) is trivial. 
We prove (ii). 
Let $\pi_1,\pi_2,\dots,\pi_p$ be the conjugate elements of $\pi$ over $\mathbb{Q}_p$. 
Then the conjugate elements of $\pi'$ are $\pi_1+u\pi_1^2,\pi_2+u\pi_2^2,\dots,\pi_p+u\pi_p^2$. 
For all $1 \leq i \leq p$, multiply the equations
\[ (ux^2 + x) - (\pi_i - u\pi_i^2) = (x - \pi_i) (ux + 1 + u \pi_i) = -u (x - \pi_i) \left\{ (-x-u^{-1}) - \pi_i \right\}, \]
then we have the result. 
\end{proof}
%%%%%%%%%%%%%%%%%%%%%%%%%%%%%%%%%%%%%%%
%%%%%%%%%%%%%%%%%%%%%%%%%%%%%%%%%%%%%%%
\begin{proposition}\label{Type0Hard}
Let $f = \sum_i a_i x^i \in E_{\mathbb{Q}_p}^p$. 
If $f$ is of type $\typezero$ and $v_p(a_1)=2$, then we have 
\[ f \sim x^p + a_0(1 - u^p a_0), \]
where we put $u=-a_1/(p a_0)$. 
\end{proposition}
\begin{proof}
By Corollary \ref{Cor:Type0TypeLambda} (ii), we have $f \sim f_1:=x^p + a_1 x + a_0$. 
Take a root $\pi$ of $f_1$. 
Let  $g_1=\sum_i b_i x^i$ be the minimal polynomial of $\pi + u \pi^2$ over $\mathbb{Q}_p$. 
By Lemma \ref{Lem:Type0} (i), we have $f_1 \sim g_1$ 
and by (ii), an equality
\[ g_1(ux^2 + x) = -u^p f_1(x) f_1(-x -u^{-1}) \]
holds. 
By comparing the coefficients of $x^0$ and $x^1$ in the both-hand side, 
we have
\[ b_1 = a_1 + up a_0 + u^{p-1} a_1^2 = u^{p-1} a_1^2,\quad b_0=a_0 + u^{p-1} a_0 a_1 - u^p a_0^2. \]
Since the type is independent of equivalence classes, $g_1$ is also of type $\typezero$. 
By the inequality $v_p(b_1) \geq 3$, Proposition \ref{Prop:Type0TypeLambda} (iii) gives  
the equivalence
\[ g_1 \sim x^p + b_0 \] 
follows. 
By assumption, we note that $v_p(a_0 a_1) \geq 3$, 
so that Proposition \ref{Prop:Type0TypeLambda} (iii) gives
the equivalence
\[ x^p + b_0 \sim x^p + a_0(1 - u^p a_0). \]
\end{proof}
%%%%%%%%%%%%%%%%%%%%%%%%%%%%%%%%%%%%%%%
The following lemma is a result in field theory:
%%%%%%%%%%%%%%%%%%%%%%%%%%%%%%%%%%%%%%%
\begin{lemma}\label{Linearity}
Let $f = X^e + a_{e-1} X^{e-1} + \cdots + a_1 X + a_0 \in E_K^e$ and 
$\pi_f$ a root of $f$. 
Then, for any $u \in U_K$, 
the Eisenstein polynomial of $u \pi_f$ over $K$ is 
\[ X^e + u a_{e-1} X^{e-1} + u^2 a_{e-2} X^{e-2} + \cdots + u^{e-1} a_1 X + u^e a_0. \]
\end{lemma}
\begin{proof}
This is trivial. 
\end{proof}
%%%%%%%%%%%%%%%%%%%%%%%%%%%%%%%%%%%%%%%
%%%%%%%%%%%%%%%%%%%%%%%%%%%%%%%%%%%%%%%
\noindent
By the following theorem, we can identify a given polynomial as the one in Table \ref{TableDegreeP}: 
%%%%%%%%%%%%%%%%%%%%%%%%%%%%%%%%%%%%%%%
\begin{theorem}\label{MainTheorem}
Let $f =\sum_i a_i x^i \in E_{\mathbb{Q}_p}^p$. 
If $a_i \not= 0$, we put $u_i = a_i/p^{s_i}$ $(s_i=v_p(a_i))$. 
For $u \in \mathbb{Z}_p^{\times}$ and $n \in \mathbb{Z}_{\geq 1}$, 
we denote by $\langle u$ mod $p^n \rangle$ the integer $i$ such that $u \equiv i \pmod{p^n}$ and $1 \leq i \leq p^n-1$. 
Put 
\[ \mathfrak{u}=\langle u_0^{-1}\ \mathrm{mod}\ p \rangle,\quad  
 a'_0=\mathfrak{u}^p a_0,\quad  
a''_0=a_0 \left\{ 1 + (u_0^{-1} u_1)^p a_0 \right\}, \]
\[ u_{\lambda}' = \mathfrak{u}^{p-\lambda} u_{\lambda}\ \mathrm{and}\  
a'_{\lambda} = u_{\lambda}' p. \] 

\noindent
$\mathrm{(i)}$ If $f$ is of type $\typelambda$, 
then 
\[
f \sim 
\begin{cases}
f_1:=X^p + \langle a_{\lambda}'\ \mathrm{mod}\ p^2 \rangle \cdot X^{\lambda} + p & 
\text{$\mathrm{if}$ $\lambda \not= p-1$ $\mathrm{or}$ 
$u_{\lambda}' \not\equiv -1$ $\mathrm{(mod}$ $p\mathrm{)}$},\\
f_2:=X^p - p X^{p-1} + \langle a'_0\ \mathrm{mod}\ p^3 \rangle & \text{$\mathrm{if}$ 
$\lambda=p-1$ $\mathrm{and}$ $u_{\lambda}' \equiv -1$ $\mathrm{(mod}$ $p\mathrm{)}$}. \\
\end{cases}
\]

\noindent
$\mathrm{(ii)}$ If $f$ is of type $\typezero$, 
then 
\[ 
f \sim 
\begin{cases}
f_3:=X^p + \langle a'_0\ \mathrm{mod}\ p^3 \rangle & \text{$\mathrm{if}$ $v_p(a_1) \not=2$}, \\
f_4:=X^p + \langle a_0''\ \mathrm{mod}\ p^3 \rangle & 
\text{$\mathrm{if}$ $v_p(a_1)=2$}. \\
\end{cases}
\]
\end{theorem}
\begin{proof}
(i) We assume that $f$ is of type $\typelambda$. 
Then we have $f \sim g_1:=x^p + a_{\lambda} x^{\lambda} +a_0$ by Corollary \ref{Cor:Type0TypeLambda} (i). 
Apply Lemma \ref{Linearity} to $g_1$, we have
$g_1 \sim g_2:= x^p + \mathfrak{u}^{p-\lambda} a_{\lambda} x^{\lambda} + \mathfrak{u}^p a_0$. 

First, we show the case $\lambda \not= p-1$ or $u'_{\lambda} \not\equiv -1$ (mod $p$). 
In case $\lambda \not= p-1$, we note that $a_0' \equiv p$ (mod $p^2$), so that we have
\[ v_p(g_2,f_1) \geq 2 > 1 + \frac{\lambda}{p-1} = u_f. \]
Hence we have $g_2 \sim f_1$ by Proposition \ref{Krasner}, and an equivalence $f \sim f_1$ follows it. 
In case $\lambda=p-1$ and $u'_{\lambda} \not\equiv -1$ (mod $p$), we note that
$\sqrt[p-1]{\lambda u'_{\lambda}} \not\in \mathbb{F}_p$. 
By Proposition \ref{Krasner} and \ref{EquivalenceAtBreak} below (as $m=1$, $\omega=\lambda u_{\lambda}'$, $\pi_K=p$),  
we have $g_2 \sim g_3:=x^p + a_{\lambda}' x^{\lambda} + p$. 
By the inequality $v_p(g_3,f_1) > 2$ and Proposition \ref{Krasner}, 
we have $g_3 \sim f_1$, so that$f \sim f_1$. 

Second, we prove the case $\lambda = p-1$ and $u_{\lambda}' \equiv -1$ (mod $p$). 
By the inequality
\[ v_p(g_3,f_2) \geq 2 + (p-1)/p > 2 = u_f \]
and Proposition \ref{Krasner}, we have $g_3 \sim f_2$. 
Hence we obtain an equivalence $f \sim f_2$. 

\noindent
(ii) Suppose that $f$ is of type $\typezero$ and $v_p(a_1) \not= 2$. 
According to Corollary \ref{Cor:Type0TypeLambda} (ii), 
we have $f \sim g_4:=X^p + a_0$. 
Apply Lemma \ref{Linearity} to $f_0$ with similar argument as 
in the proof of (i), 
then we have $g_4 \sim g_5:=X^p + a'_0$. 
Note that $v_p(g_5,f_3) \geq 3 > u_f$, 
so that Proposition \ref{Krasner} gives $g_5 \sim f_3$. 
Thus we deduce the desired equivalence $f \sim f_3$. 

Second, we suppose that $f$ is of type $\typezero$ and $v_p(a_1)=2$. 
By Proposition \ref{Type0Hard}, we have 
$f \sim g_6:=X^p + a''_0$. 
Note that $v_p(g_6,f_4) \geq 3 > u_f$, thus Proposition \ref{Krasner} 
shows $g_6 \sim f_4$. 
Therefore, we have $f \sim f_4$. 
\end{proof}
%%%%%%%%%%%%%%%%%%%%%%%%%%%%%%%%%%%%%%%%
\begin{proposition}[\cite{Amano71}, Prop.\ 5]\label{EquivalenceAtBreak}
Consider two polynomials 
\[ f = x^p + s p^m x^{p-1} + t p\ \mathrm{and}\    
g = x^p + s p^m x^{p-1} + t' p \quad (s,t,t' \in \mathbb{Z}_p^{\times}) \]
in $E_{\mathbb{Q}_p}^p$. 
Suppose $t \equiv t' \equiv 1$ $\mathrm{(mod}$ $p\mathrm{)}$. 
If $v_p(t - t') = u_f-1$ and $s \not\equiv -1$ $\mathrm{(mod}$ $p\mathrm{)}$, then we have $f \sim g$. 
\end{proposition}
%%%%%%%%%%%%%%%%%%%%%%%%%%%%%%%%%%%%%%%%

%%%%%%%%%%%%%%%%%%%%%%%%%%%%%%%%%%%%%%%%%%%%%%%%%%%%%%%%%%%%%%%%%%%%%%%%%%%%%%%%%%%%%%
\subsection{Appendix: An algorithm computing the ramification breaks}

Let $K$ be a finite extension of $\mathbb{Q}_p$ and $f$ an Eisenstein polynomial over $K$. 
Put $H=\mathrm{Hom}_K(L_f,\overline{K})$. 
Let $H_i$, $H^u$ be the lower and upper numbering ramification sets of $H$ in the sense of \cite{Deligne84}, Appendice. 
We give an algorithm for computing the breaks of $H_i$ and $H^u$. 
If $L_f/K$ is a Galois extension, then the ramification sets coincides with the ramification groups in the sense of \cite{Serre79}. 
Let $i_1,i_2,\dots,i_m$ (resp.\ $u_1,u_2,\dots,u_m$) be the lower (resp.\ upper) numbering ramification breaks. 
\begin{algorithm}
Input: $K$, $f$

\noindent
Output: $\{ (i_1,\# H_{i_1}), (i_2,\# H_{i_2}),\dots,(i_m,\# H_{i_m}) \}$ 

\noindent
$\bullet$ Let $N$ be the Newton polygon of $f(x + \pi_f) \in L_f[x]$. 

\noindent
$\bullet$ Let $s_1 > s_2 > \cdots > s_m$ be the slopes of $N$. 

\noindent
$\bullet$ Let $1 < x_1 < x_2 < \cdots < x_m$ be the $x$-coordinates of the vertexes of $N$. 

\noindent
$\bullet$ Return $\{ (-s_1 - 1,x_1), (-s_2 - 1,x_2),\dots,(-s_m - 1,x_m) \}$. 
\end{algorithm}
\begin{remark}
The Newton polygon can be computed by Magma as the inner function 
\texttt{NewtonPolygon}($h(x)$). 
However, in fact, the valuations of the coefficients of $f(x + \pi_f) = \sum_{i=1}^e b_i x^i$ 
can be directly written as 
\[ v_{L_f}(b_i) = \min_{i\leq j\leq e}\{ v_{L_f}(a_j) + v_{L_f}(\binom{j}{i})+ (j-i) \}, \]
where we write $f = \sum_{i=0}^e a_i x^i$. 
\end{remark}
\begin{algorithm}
Input: $\{ (i_1,\# H_{i_1}), (i_2,\# H_{i_2}),\dots,(i_m,\# H_{i_m}) \}$ 

\noindent
Output: $\{ u_1, u_2,\dots,u_m \}$ 

\noindent
$\bullet$ Put $u_1=i_1$.

\noindent
$\bullet$ If $m = 1$ then:

\noindent
\quad $\bullet$ Return $\{ u_1 \}$.

\noindent
$\bullet$ If $m \geq 2$ then:

\noindent
\quad $\bullet$
$S \leftarrow \{ u_1 \}$.

\noindent
\quad $\bullet$ $s \leftarrow 2$.

\noindent
\quad $\bullet \ $ While $s \leq m$:

\noindent
\quad \quad $\bullet$ $u_s=(i_s-i_{s-1})\# H_{i_s}/\# H_{i_1}+u_{s-1}$.

\noindent
\quad \quad $\bullet$ $s \leftarrow s+1$. 

\noindent
\quad \quad $\bullet$ $S \leftarrow S \cup \{ u_s \}$.

\noindent
\quad $\bullet$ Return $S$.
\end{algorithm}

%%%%%%%%%%%%%%%%% 参考文献 %%%%%%%%%%%%%%%%%%%%%%%%%%%

\end{document}